\documentclass[a4paper,11pt]{amsart}
\usepackage[colorlinks, linkcolor=blue,anchorcolor=Periwinkle,
    citecolor=blue,urlcolor=Emerald]{hyperref}
\usepackage[all]{xy}
\SelectTips{cm}{}

\usepackage{graphicx}
\usepackage{psfrag}

\usepackage{mathtools}
\usepackage{tikz}
\usepackage{tikz-cd}
\usepackage{adjustbox}
\usepackage{extarrows}

\usepackage{amssymb}
\usetikzlibrary{decorations.pathreplacing}
\usetikzlibrary{matrix,arrows}

\usepackage{enumitem}

\usepackage{stmaryrd}

\newcommand{\ie}{{i.e.}\ }
\newcommand{\cf}{{cf.}\ }

\newcommand{\ko}{\: , \;}

\numberwithin{equation}{subsection}

\newtheorem{classification-theorem}[subsection]{Classification Theorem}
\newtheorem{decomposition-theorem}[subsection]{Decomposition Theorem}
\newtheorem{proposition-definition}[subsection]{Proposition-Definition}
\newtheorem{periodicity-conjecture}[subsection]{Periodicity Conjecture}

\newtheorem{theorem}{Theorem}
\numberwithin{theorem}{subsection}
\newtheorem{thmx}{Theorem}

\newtheorem{lemma}[theorem]{Lemma}
\newtheorem{proposition}[theorem]{Proposition}
\newtheorem{corollary}[theorem]{Corollary}
\newtheorem{corx}[thmx]{Corollary}

\theoremstyle{definition}

\theoremstyle{plain}

\newtheorem{remarks}[theorem]{Remarks}

\newcommand{\reminder}[1]{}

\renewcommand{\mod}{\mathrm{mod}\,}

\newcommand{\sg}{\mathrm{sg}\,}

\newcommand{\Gproj}{\mathrm{Gproj}\,}

\newcommand{\per}{\mathrm{per}\,}
\newcommand{\pvd}{\mathrm{pvd}\,}

\newcommand{\op}{^{op}}

\newcommand{\HH}{\mathrm{HH}}
\newcommand{\HC}{\mathrm{HC}}
\newcommand{\HN}{\mathrm{HN}}

\newcommand{\Z}{\mathbb{Z}}

\newcommand{\iso}{\xrightarrow{_\sim}}
\newcommand{\liso}{\xleftarrow{_\sim}}
\newcommand{\id}{\mathbf{1}}

\newcommand{\Hom}{\mathrm{Hom}}

\newcommand{\RHom}{\mathrm{RHom}}

\newcommand{\ten}{\otimes}
\newcommand{\lten}{\overset{\mathrm{L}}{\ten}}

\newcommand{\ca}{{\mathcal A}}
\newcommand{\cb}{{\mathcal B}}
\newcommand{\cc}{{\mathcal C}}
\newcommand{\cd}{{\mathcal D}}

\newcommand{\bp}{\mathbf{p}}

\newcommand{\rb}{\mathrm{B}}
\newcommand{\rc}{\mathrm{C}}

\newcommand{\La}{\Lambda}
\newcommand{\Si}{\Sigma}

\newcommand{\wg}{\wedge}

\newcommand{\eps}{\varepsilon}
\renewcommand{\phi}{\varphi}

\newcommand{\Supp}{\mathrm{Supp}\,}
\newcommand{\Max}{\mathrm{Max}}
\newcommand{\Sing}{\mathrm{Sing}}

\renewcommand{\tilde}[1]{\widetilde{#1}}

\begin{document}

\date{\today}

\title[Calabi--Yau structures for symmetric orders]{Calabi--Yau structures on derived and singularity\\[0.15cm] categories of symmetric orders}

\author{Norihiro Hanihara}
\address{Faculty of Mathematics, Kyushu University, 744 Motooka, Nishi-ku, Fukuoka, 819-0395, Japan}
\address{Kavli Institute for the Physics and Mathematics of the Universe (WPI), The University of Tokyo, 5-1-5 Kashiwanoha, Kashiwa, Chiba, 277-8583, Japan}
\email{hanihara@math.kyushu-u.ac.jp}

\author{Junyang Liu}
\address{School of Mathematical Sciences, University of Science and Technology of China, Hefei 230026, China}
\email{liuj@imj-prg.fr}
\urladdr{https://webusers.imj-prg.fr/~junyang.liu}

\begin{abstract}
We construct left and right Calabi--Yau structures on derived respectively singularity categories of symmetric orders $\La$ over commutative Gorenstein rings $R$.
For this, we first construct Calabi--Yau structures over $R$ by lifting Amiot's construction of Calabi--Yau structures on Verdier quotients to the dg level. Then we prove base change properties relating Calabi--Yau structures over $R$ to those over the base field $k$.
As a result, we prove the existence of a right Calabi--Yau structure on the dg singularity category associated with $\La$ which is a cyclic lift of the weak Calabi--Yau structure constructed by the first-named author and Iyama.
We also show the existence of a left Calabi--Yau structure on the dg bounded derived category of $\La$. This is a non-commutative generalization of a result by Brav and Dyckerhoff.
By combining the existence of the right Calabi--Yau structure on the dg singularity category with a structure theorem by Keller and the second-named author, we deduce that under suitable hypotheses, the singularity category associated with $\La$ is triangle equivalent to a generalized cluster category in the sense of Amiot.
\end{abstract}




\maketitle

\vspace*{-1cm}
\tableofcontents

\section{Introduction}

Calabi--Yau triangulated categories are ubiquitous and of great interest in various areas of mathematics. They play an important role in algebraic geometry (notably in mirror symmetry), singularity theory, representation theory, and many other subjects. Let $k$ be a field and $d$ an integer. Recall that a $k$-linear $\Hom$-finite triangulated category $\cc$ is {\em $d$-Calabi--Yau} if it is endowed with bifunctorial bijections
\[
\Hom_{\cc}(X, Y) \xlongrightarrow{_\sim} D\Hom_{\cc}(Y, \Si^d X) \ko X, Y \in \cc \: ,
\]
where $D$ denotes the duality over $k$.

On the other hand, it is well-known that one should work with enhancements of triangulated categories, which provide us with additional tools for comparing such categories. Here we employ dg(=differential graded) enhancements \cite{Keller06d}. Following \cite{Yeung16, BravDyckerhoff19, KellerWang23}, we use (left and right) Calabi--Yau structures on a dg category to enhance the Calabi--Yau property of a triangulated category. We recall them in section~\ref{ss:Calabi-Yau structures}.

In this article, we focus on Calabi--Yau structures appearing in Cohen--Macaulay representation theory \cite{Yoshino90, LeuschkeWiegand12}, namely, the derived and singularity categories of commutative Gorenstein rings, or more generally, symmetric orders. Let $\La$ be a noetherian ring. The {\em singularity category} \cite{Buchweitz21, Orlov04} associated with $\La$ is, by definition, the Verdier quotient
\[
\sg \La = \cd^b(\mod \La)/\per \La
\]
of the bounded derived category of finitely generated $\La$-modules by the perfect derived category of $\La$.

Let us first discuss the singularity category associated with a symmetric order. The classical Auslander--Reiten duality \cite{Auslander78} implies that for a symmetric order $\La$ of Krull dimension $d$ which has an isolated singularity, the associated singularity category $\sg \La$ is a $(d-1)$-Calabi--Yau triangulated category. The first aim of this paper is to construct a right Calabi--Yau structure on the canonical dg enhancement of $\sg \La$, enhancing Auslander--Reiten's Calabi--Yau structure.

The first main result gives a general framework for constructing Calabi--Yau structures on dg quotients. Essentially, we work with dg categories over a commutative ring $R$. In section~\ref{ss:a dg enhancement of Amiot's construction}, we introduce a construction which lifts Amiot's approach to the construction of Serre functors on Verdier quotients to the dg level. Inspired by \cite{HaniharaIyama22, KellerLiu23b}, we prove that this construction can be interpreted as applying the connecting morphism in dual Hochschild homology. Moreover, we show that the connecting morphism preserves non-degeneracy under mild assumptions.

\begin{thmx}[see Theorem~\ref{thm:connecting morphism} for details] \label{thm:A}
Let $R$ be a commutative ring. Let $\ca$, $\cb$, and $\cc$ be small dg $R$-categories and
\[
\begin{tikzcd}
0 \arrow{r} & \cb \arrow{r}{I} & \ca \arrow{r}{Q} & \cc \arrow{r} & 0
\end{tikzcd}
\]
an exact sequence of dg categories. Let $\nu$ be a dg autoequivalence of $\ca$ which maps $I\cb$ to itself. Let $M$ be a dg $\ca$-bimodule which is isomorphic to $\ca(\nu ?, -)$ in $\cd(\ca^e)$. Denote $\mathrm{R}I^e_* M$ by $M_\cb$ and $\mathrm{L}Q^{e*}M$ by $M^\cc$. Denote the functor $\RHom_R(?, R)$ by $D$.
\begin{itemize}
\item[a)] The connecting morphism
\[
\delta \colon H^0(D\HH(\cb, M_\cb)) \longrightarrow H^1(D\HH(\cc, M^\cc))\: ,
\]
is an enhancement of Amiot's construction, \cf section~\ref{ss:a dg enhancement of Amiot's construction}.
\item[b)] Suppose that the canonical morphism $\ca \to DD\ca$ in $\cd(\ca^e)$ is an isomorphism. If the canonical morphism $\ca \to \RHom_{\cb}(\ca, \ca)$ in $\cd(\ca^e)$ is an isomorphism, then the connecting morphism $\delta$ preserves non-degeneracy.
\end{itemize}
\end{thmx}

In particular, we lift Amiot's construction of Serre functors for certain Verdier quotients and the criterion of non-degeneracy in Theorem~1.3 of~\cite{Amiot09} to the level of dg categories, \cf~ Remarks~\ref{rk:lift}. Lemma~\ref{lem:local cover 1} shows that our condition on the morphism $\ca \to \RHom_{\cb}(\ca, \ca)$ discussed also in \cite{HaniharaKeller25} is a suitable dg version of the `local cover' condition in Amiot's criterion. We refer to Lemma~\ref{lem:local cover 2} for an alternative equivalent condition.

We apply Theorem~\ref{thm:A} to the singularity categories associated with symmetric orders. Let $k$ be a field and $R$ a commutative Gorenstein $k$-algebra of Krull dimension $d$. Denote by $\Max^k_d R$ the set of maximal ideals of $R$ of height at least $d$ whose residue fields are finite-dimensional over $k$. Let $\La$ be a symmetric $R$-order whose singular locus is contained in $\Max^k_d R$. Applying Theorem~\ref{thm:A} to $\ca=\cd^b_{dg}(\mod \La)$, $\cb=\per\!_{dg}\La$, $\cc=\sg\!_{dg}\La$, and $M=\ca$, we deduce in Proposition~\ref{prop:right CY structure} that the dg singularity category $\sg\!_{dg}\La$ has a right $(-1)$-Calabi--Yau structure over $R$. In Proposition~\ref{prop:right base change}, we show a base change property which relates (weak) right Calabi--Yau structures over $R$ to those over $k$. By combining Propositions~\ref{prop:right CY structure} and \ref{prop:right base change} we obtain the following result.

\begin{thmx}[=Theorem~\ref{thm:right CY structure}] \label{thm:B}
Let $k$ be a field and $R$ a commutative Gorenstein $k$-algebra of Krull dimension $d$. Let $\La$ be a symmetric $R$-order satisfying $\Sing_R \La \subseteq \Max_d^k R$. Then the dg category $\sg\!_{dg}\La$ carries a right $(d-1)$-Calabi--Yau structure over $k$.
\end{thmx}

This genuine right $(d-1)$-Calabi--Yau structure over $k$ on the dg category $\sg\!_{dg}\La$ is a lift of the weak right $(d-1)$-Calabi--Yau structure over $k$ constructed in part~(2) of Theorem~3.6 of~\cite{HaniharaIyama22}.

Recall that if $A$ is a smooth connective dg $k$-algebra endowed with a left $d$-Calabi--Yau structure in the sense of Brav--Dyckerhoff \cite{BravDyckerhoff19}, then its associated (generalized) cluster category in the sense of Amiot \cite{Amiot09} is the idempotent completion of the Verdier quotient $\per A/\pvd A$ of the perfect derived category of $A$ by its thick subcategory $\pvd A$ whose objects are the dg $A$-modules with perfect underlying dg $k$-modules. By combining Theorem~\ref{thm:B} with the structure theorem of \cite{KellerLiu23b} we establish a triangle equivalence between the singularity category $\sg \La$ and the cluster category associated with a deformed dg preprojective algebra.

\begin{corx}[=Corollary~\ref{cor:right CY structure}] \label{cor:C}
Let $k$ be a field of characteristic $0$ and $(R, \mathfrak{m})$ a complete commutative Gorenstein local $k$-algebra of Krull dimension $d\geq 2$. Let $\La$ be a symmetric $R$-order satisfying $\Sing_R \La \subseteq \{\mathfrak{m}\}$. Suppose that the triangulated category $\sg \La$ contains a $(d-1)$-cluster-tilting object. Then there exists a $d$-dimensional deformed dg preprojective algebra $\Pi$ such that $\sg \La$ is triangle equivalent to the associated cluster category $\cc_\Pi$.
\end{corx}

From the perspective of realizing the singularity category as the cluster category associated with a deformed dg preprojective algebra, Corollary~\ref{cor:C} concludes a development which started with the result of Keller--Reiten \cite{KellerReiten08}. They proved that the singularity category associated with the invariant algebra of $\mathbb{C}\llbracket x, y, z \rrbracket$ under the diagonal $\Z/3\Z$-action is triangle equivalent to the cluster category associated with the $3$-Kronecker quiver.
This was generalized to suitable higher-dimensional cyclic quotient singularities by Thanhoffer de V\"olcsey--Van~den Bergh \cite{ThanhofferVandenBergh16} (based on the work of Amiot--Iyama--Reiten \cite{AmiotIyamaReiten15}), who obtained the above dg algebra $A$ using Van den Bergh's superpotential theorem \cite{VandenBergh15}.
A generalization to suitable $3$-dimensional non-cyclic quotient singularities is due to Kalck--Yang \cite{KalckYang18}, who relied on Ginzburg's \cite{Ginzburg06} to obtain the required $3$-dimensional Ginzburg dg algebra $A$. Most recently, the second-named author generalized \cite{Liu26} their result to arbitrary dimensions. 
Corollary~\ref{cor:C} is a significant generalization (respectively, refinement) of these results insofar as we prove that $A$ is a deformed dg preprojective algebra.

We refer to \cite{KalckYang18, Hanihara22} for general results on realizing an algebraic triangulated category as the Verdier quotient $\per A/\pvd A$ for a smooth connective dg algebra $A$ (not necessarily a deformed dg preprojective algebra). We also refer to \cite{Iyama18, HaniharaIyama22} for a different approach to constructing triangle equivalences between singularity categories and cluster categories.

We also study left Calabi--Yau structures on the dg bounded derived categories of symmetric orders. As in the case of right Calabi--Yau structures, we give a base change property which relates weak left Calabi--Yau structures over $R$ to those over $k$ in Proposition~\ref{prop:left base change}. It leads to the following result.

\begin{thmx}[=Theorem~\ref{thm:left CY structure}]
Let $k$ be a perfect field and $R$ a finitely generated commutative Gorenstein $k$-algebra of Krull dimension $d$. Let $\La$ be a symmetric $R$-order. Then the dg category $\cd^b_{dg}(\mod \La)$ carries a left $d$-Calabi--Yau structure over $k$.
\end{thmx}

This is a generalization of the result in Proposition~5.12 of \cite{BravDyckerhoff19}, which shows that the dg category $\cd^b_{dg}(\mod R)$ carries a left $d$-Calabi--Yau structure, extending the result from commutative Gorenstein algebras to symmetric orders. As a consequence, the dg quotient $\sg\!_{dg}\La$ also carries a left $d$-Calabi--Yau structure over $k$, \cf Corollary~\ref{cor:left CY structure}.

\subsection*{Acknowledgments}

The authors are grateful to Bernhard Keller for valuable comments and suggestions.

The first-named author is supported by JSPS KAKENHI Grant Numbers JP22KJ0737 and JP25K17233.

\section{Notation}

The following notation is used throughout the article: We let $k$ be a field. Algebras have units and morphisms of algebras preserve the units. Modules are unital right modules. For a commutative ground ring $R$, we assume that $R$ acts centrally on all bimodules we consider. For a ring $A$, we denote the category of finitely generated $A$-modules by $\mod A$. The degree of a homogeneous element $a$ in a graded vector space is denoted by $|a|$. We denote the shift functor of graded vector spaces by $\Si$. We use cohomological grading so that differentials are of degree $1$. For a commutative ground ring $R$ and a dg $R$-algebra $A$, we write $A^e$ for the dg enveloping algebra $A\ten_R A\op$ and denote the functors $\RHom_R(?, R)$ and $\RHom_{A^e}(?, A^e)$ by $D$ respectively $\vee$. For an algebraic triangulated category $\cc$, we write $\cc_{dg}$ for its canonical dg enhancement.

\section{Preliminaries}

\subsection{Singularity categories}

For a noetherian ring $\La$, we write $\cd^b(\mod \La)$ for the bound\-ed derived category of finitely generated $\La$-modules. Its thick subcategory generated by the free $\La$-module of rank one is the {\em perfect derived category} $\per \La$. The {\em singularity category} associated with $\La$ is defined \cite{Buchweitz21, Orlov04} as the Verdier quotient
\[
\sg \La = \cd^b(\mod \La)/\per \La \: .
\]
When $\La$ is a module-finite algebra over a commutative noetherian ring, the singularity category $\sg \La$ vanishes if and only if each $\La$-module has finite projective dimension. A commutative noetherian ring of finite Krull dimension is {\em Gorenstein} if it has finite injective dimension as a module over itself. Let $R$ be a commutative Gorenstein ring. A module-finite $R$-algebra $\La$ is a {\em symmetric $R$-order} if we have an isomorphism $\La \to \RHom_R(\La, R)$ in $\cd(\La^e)$.

\subsection{Calabi--Yau structures} \label{ss:Calabi-Yau structures}

In this section, we recall the necessary background on Hochschild and cyclic homologies, left and right Calabi--Yau structures. Fix a commutative ring $R$ as the ground ring. We work in the setting of dg $R$-algebras but everything generalizes to the setting of small dg $R$-categories.

Following section~1 of \cite{Kassel87}, a {\em mixed complex} over $R$ is a dg module over the dg algebra $\La=R[t]/(t^2)$, where $t$ is an indeterminate of degree $-1$ satisfying $d(t)=0$. For a dg $R$-algebra $A$ which is flat as a dg $R$-module, its {\em mixed complex} $\mathrm{M}(A)$ is defined as follows. Its underlying complex is defined to be the cone of the map $\id-\tau$ from the sum total complex $\rb^+(A)$ of
\[
\begin{tikzcd}
  \cdots\arrow{r} & A^{\ten_R 3} \arrow{r}{b'} & A^{\ten_R 2} \arrow{r}{b'} & A
\end{tikzcd}
\]
to the sum total complex $\rc(A)$ of
\[
\begin{tikzcd}
  \cdots\arrow{r} & A^{\ten_R 3} \arrow{r}{b} & A^{\ten_R 2} \arrow{r}{b} & A \: .
\end{tikzcd}
\]
Here $\tau$ maps $a_1\ten\cdots\ten a_p$ to
\[
(-1)^{(|a_p|+1)(p-1+|a_1|+\cdots+|a_{p-1}|)}a_p\ten a_1\ten \cdots\ten a_{p-1}\: ,
\]
the differential of $A^{\ten_R p}$ maps $a_1\ten\cdots\ten a_p$ to
\[
\sum_{i=1}^p (-1)^{i-1+|a_1|+\cdots+|a_{i-1}|}a_1\ten\cdots\ten d(a_i)\ten\cdots\ten a_p \: ,
\]
the map $b$ is the differential of the Hochschild chain complex and $b'$ is induced by that of the augmented bar resolution. Explicitly, the differential $b$ maps $a_1\ten\cdots\ten a_p$ to
\begin{align*}
& \sum_{i=1}^{p-1}(-1)^{i-1+|a_1|+\cdots+|a_i|} a_1\ten\cdots\ten a_i a_{i+1}\ten\cdots\ten a_p \\
& +(-1)^{(|a_p|+1)(p+|a_1|+\cdots+|a_{p-1}|)-1}a_p a_1\ten\cdots\ten a_{p-1}
\end{align*}
and $b'$ maps $a_1\ten\cdots\ten a_p$ to 
\[
\sum_{i=1}^{p-1}(-1)^{i-1+|a_1|+\cdots+|a_i|}a_1\ten\cdots\ten a_i a_{i+1}\ten\cdots\ten a_p\: .
\]
The $\La$-module structure on $\mathrm{M}(A)$ is determined by the action of $t$, which vanishes on $\rb^+(A)$ and maps the component $A^{\ten_R p}$ of $\rc(A)$ to the corresponding component of $\rb^+(A)$ via the map $\sum_{i=0}^{p-1}\tau^i$.

The {\em Hochschild complex} $\HH(A)$ of $A$ is defined to be the underlying complex of $\mathrm{M}(A)$. By construction, we have the canonical triangle
\[
\begin{tikzcd}
\rb^+(A) \arrow{r}{\id-\tau} & \rc(A) \arrow{r} & \mathrm{M}(A) \arrow{r} & \Si \rb^+(A)
\end{tikzcd}
\]
in $\cd(R)$. The complex $\rb^+(A)$ is contractible (since it is the sum total complex of a contractible complex of complexes) so that the morphism $\rc(A) \to \mathrm{M}(A)$ is a quasi-isomorphism. This shows that our definition of the Hochschild complex coincides with the classical one up to a canonical quasi-isomorphism. For a dg $A$-bimodule $M$, the {\em Hochschild complex of $A$ with coefficients in $M$} denoted by $\HH(A,M)$ is defined to be the sum total complex of
\[
\begin{tikzcd}
  \cdots\arrow{r} & M\ten_R A^{\ten_R 2} \arrow{r}{b} & M\ten_R A \arrow{r}{b} & M \: ,
\end{tikzcd}
\]
where the differential of $M\ten_R A^{\ten_R p}$ maps $m\ten a_1\ten\cdots\ten a_p$ to
\[
d(m)\ten a_1 \ten\cdots\ten a_p +\sum_{i=1}^p (-1)^{i+|m|+|a_1|+\cdots+|a_{i-1}|}m\ten a_1 \ten\cdots\ten d(a_i)\ten\cdots\ten a_p
\]
and the differential $b$ maps $m\ten a_1\ten\cdots\ten a_p$ to
\begin{align*}
& (-1)^{|m|}ma_1\ten\cdots\ten a_p \\
& +\sum_{i=1}^{p-1}(-1)^{i+|m|+|a_1|+\cdots+|a_i|} m\ten a_1\ten\cdots\ten a_i a_{i+1}\ten\cdots\ten a_p \\
& +(-1)^{(|a_p|+1)(p+1+|m|+|a_1|+\cdots+|a_{p-1}|)-1}a_p m\ten a_1\ten\cdots\ten a_{p-1}\: .
\end{align*}
The homologies of $\HH(A)$ and $\HH(A,M)$ are called {\em Hochschild homology} $\HH_*(A)$ respectively {\em Hochschild homology with coefficients in $M$} denoted by $\HH_*(A,M)$.

Let $\bp R$ be the minimal cofibrant resolution of $R$ as a dg $\La$-module. The {\em cyclic complex} $\HC(A)$ of $A$ is defined to be the complex $\mathrm{M}(A)\ten_\La \bp R$. The {\em negative cyclic complex} $\HN(A)$ of $A$ is defined to be the complex $\Hom_{\cc_{dg}(\La)} (\bp R, \mathrm{M}(A))$. Their homologies are called {\em Hochschild homology} $\HH_*(A)$, {\em cyclic homology} $\HC_*(A)$, {\em negative cyclic homology} $\HN_*(A)$, respectively. We have the ISB triangle
\[
\begin{tikzcd}
  \HH(A)\arrow{r}{I} & \HC(A)\arrow{r}{S} & \Si^2\HC(A)\arrow{r}{B} & \Si \HH(A)
\end{tikzcd}
\]
in the homotopy category of complexes which relates Hochschild and cyclic complexes.

We write $\cd(A)$ for the (unbounded) derived category of $A$. Its thick subcategory generated by the free dg $A$-module of rank one is the {\em perfect derived category} $\per A$. It consists of compact objects in $\cd(A)$. Recall that $A$ is {\em smooth} if $A$ is perfect over $A^e$.

We now recall the notions of Calabi--Yau structures from section~3 of the original article \cite{BravDyckerhoff19}, section~4 of the article \cite{Yeung16} or section~10.3 of the survey article \cite{KellerWang23}. Fix an integer $d$. Let $A$ be a smooth dg $R$-algebra. A {\em left $d$-Calabi--Yau structure} on $A$ is a class $[\tilde{\xi}]$ in $\HN_d(A)$ whose image $[\xi]$ under the canonical map $\HN_d(A)\to \HH_d(A)$ is {\em non-degenerate}, \ie the morphism $\Si^d A^\vee \to A$ in $\cd(A^e)$ obtained from $[\xi]$ via
\[
\begin{tikzcd}
\HH_d(A)\arrow{r}{\sim} & H^{-d}(\RHom_{A^e}(A^\vee,A))\arrow[no head]{r}{\sim} & \Hom_{\cd(A^e)}(\Si^d A^\vee,A)
\end{tikzcd}
\]
is an isomorphism.

Let $A$ be a dg $R$-algebra. A {\em right $d$-Calabi--Yau structure} on $A$ is a class $[\tilde{x}]$ in \linebreak $H^{-d}(D\HC(A))$ whose image $[x]$ under the canonical map $H^{-d}(D\HC(A))\to H^{-d}(D\HH(A))$ is {\em non-degenerate}, \ie the morphism $A\to \Si^{-d}DA$ in $\cd(A^e)$ obtained from $[x]$ via 
\[
\begin{tikzcd}
H^{-d}(D\HH(A))\arrow[no head]{r}{\sim} & H^{-d}(\RHom_{A^e}(A, DA))\arrow[no head]{r}{\sim} & \Hom_{\cd(A^e)}(A,\Si^{-d}DA)
\end{tikzcd}
\]
is an isomorphism.

For an arbitrary dg $R$-algebra (not necessarily being flat as a dg $R$-module), the above notions are defined as applying the corresponding constructions to a quasi-isomorphic dg $R$-algebra which is flat as a dg $R$-module.

Let $k$ be a field as the ground ring. Recall that a $k$-linear category is $\Hom$-finite if all morphism spaces between its objects are finite-dimensional over $k$. A $k$-linear $\Hom$-finite triangulated category $\cc$ is {\em $d$-Calabi--Yau} if it is endowed with bifunctorial bijections
\[
\Hom_{\cc}(X, Y) \xlongrightarrow{_\sim} D\Hom_{\cc}(Y, \Si^d X) \ko X, Y \in \cc \: ,
\]
By definition, a right $d$-Calabi--Yau structure on a pretriangulated dg $k$-category $\ca$ yields a $d$-Calabi--Yau structure on the triangulated category $H^0(\ca)$.

\section{Morphisms of triangles as connecting morphisms}

We use the following setting throughout this section. Let $R$ be a commutative ring. We consider it as the ground ring. We work in the setting of dg $R$-categories which \linebreak are flat as dg $R$-modules but everything generalizes to the setting of arbitrary dg \linebreak $R$-categories (not necessarily being flat as dg $R$-modules). In the general case, we replace these dg $R$-categories with quasi-equivalent dg $R$-categories which are flat as dg $R$-modules, \cf~part~3) of Proposition 2.3 of \cite{Toen07}. Let $\ca$, $\cb$, and $\cc$ be small dg $R$-categories which are flat as dg $R$-modules and
\begin{equation} \label{eq:exact sequence of dg categories}
\begin{tikzcd}
0 \arrow{r} & \cb \arrow{r}{I} & \ca \arrow{r}{Q} & \cc \arrow{r} & 0
\end{tikzcd}
\end{equation}
an exact sequence of dg categories, \ie the sequence
\[
\begin{tikzcd}
0 \arrow{r} &\cd(\cb) \arrow{r}{\mathrm{L}I^*} & \cd(\ca) \arrow{r}{\mathrm{L}Q^*} & \cd(\cc) \arrow{r} & 0
\end{tikzcd}
\]
of triangulated categories is exact. Let $\nu$ be a dg autoequivalence of $\ca$ which maps $I\cb$ to itself. Let $M$ be a dg $\ca$-bimodule which is isomorphic to $\ca(\nu ?, -)$ in $\cd(\ca^e)$. Denote $\mathrm{R}I^e_* M$ by $M_\cb$ and $\mathrm{L}Q^{e*}M$ by $M^\cc$.

\subsection{Connecting morphisms in dual Hochschild homology}

From Theorem~3.1 of \cite{Keller98}, we deduce that the exact sequence~(\ref{eq:exact sequence of dg categories}) of dg categories yields the triangle
\begin{equation} \label{eq:triangle of Hochschild complexes}
\begin{tikzcd}
\HH(\cb, M_\cb) \arrow{r} & \HH(\ca, M) \arrow{r} & \HH(\cc, M^\cc) \arrow{r} & \Si \HH(\cb, M_\cb)
\end{tikzcd}
\end{equation}
in $\cd(R)$. Notice that this triangle can be obtained as follows. It is well-known that the exact sequence~(\ref{eq:exact sequence of dg categories}) of dg categories gives rise to the triangle
\begin{equation} \label{eq:triangle of resolutions 1}
\begin{tikzcd}
\ca \lten_{\cb} \ca \arrow{r} & \ca \arrow{r} & \cc \arrow{r} & \Si \ca \lten_{\cb} \ca
\end{tikzcd}
\end{equation}
in $\cd(\ca^e)$. Similarly, we have the triangle
\begin{equation} \label{eq:triangle of resolutions 2}
\begin{tikzcd}
\ca \lten_{\cb}M_\cb \lten_{\cb} \ca \arrow{r} & M \arrow{r} & M^\cc \arrow{r} & \Si \ca \lten_{\cb}M_\cb \lten_{\cb} \ca
\end{tikzcd}
\end{equation}
in $\cd(\ca^e)$. Applying the triangle functor $M \lten_{\ca^e}\, ?$ to the triangle~(\ref{eq:triangle of resolutions 1}) yields the triangle~(\ref{eq:triangle of Hochschild complexes}).

If we dualize the triangle~(\ref{eq:triangle of Hochschild complexes}) and take homology, we obtain the long exact sequence
\[
H^{p-1}(D\HH(\cb, M_\cb)) \xlongrightarrow{\delta} H^p(D\HH(\cc, M^\cc)) \longrightarrow H^p(D\HH(\ca, M)) \longrightarrow H^p(D\HH(\cb, M_\cb))
\]
in dual Hochschild homology with the connecting morphism $\delta$.

\subsection{A dg enhancement of Amiot's construction} \label{ss:a dg enhancement of Amiot's construction}

The aim of this section is to give a construction which assigns a morphism $\Delta(f) \colon M^\cc \to \Si D\cc$ in $\cd(\cc^e)$ to a given morphism $f\colon M_\cb \to D\cb$ in $\cd(\cb^e)$. In fact, this construction will give an interpretation of the one given in \cite{HaniharaIyama22} in terms of dual Hochschild homology.

We start with the following observation. It shows that the given morphism $f$ in $\cd(\cb^e)$ canonically yields certain morphisms in $\cd(\ca^e)$.

\begin{lemma} \label{lem:adjunctions}
Let $I\colon \cb \to \ca$ be a quasi-fully faithful dg functor. Then we have the canonical bijections
\[
\begin{tikzcd}
\phi \colon \Hom_{\cd(\cb^e)}(M_\cb, D\cb) \arrow[no head]{r}{\sim} & \Hom_{\cd(\ca^e)}(M, D(\ca \lten_{\cb} \ca))
\end{tikzcd}
\]
and
\[
\begin{tikzcd}
\psi \colon \Hom_{\cd(\ca^e)}(\ca \lten_{\cb}M_\cb \lten_{\cb} \ca, D\ca) \arrow[no head]{r}{\sim} & \Hom_{\cd(\cb^e)}(M_\cb, D\cb) \: .
\end{tikzcd}
\]
\end{lemma}

\begin{proof}
Since the dg functor $I$ is quasi-fully faithful, the restriction of the dg $\ca^e$-module $\ca$ to $\cb^e$ is isomorphic to $\cb$ in $\cd(\cb^e)$. Then the bijectivity of $\phi$ follows from the adjunction
\[
(?\lten_{\ca^e} \ca^e, \RHom_{\cb^e}(\ca^e, ?))
\]
induced by the dg $\ca^e$-$\cb^e$-bimodule $\ca^e$ and the bijectivity of $\psi$ follows from the adjunction
\[
(?\lten_{\cb^e}\ca^e, \RHom_{\ca^e}(\ca^e, ?))
\]
induced by the dg $\cb^e$-$\ca^e$-bimodule $\ca^e$.
\end{proof}

Using these bijections, our construction is summarized as follows.

\begin{proposition} \label{prop:morphism of triangles}
Consider the diagram
\begin{equation} \label{eq:morphism of triangles}
\begin{tikzcd}
\ca \lten_{\cb}M_\cb \lten_{\cb}\ca \arrow{r} \arrow{d}{\psi^{-1}(f)} & M \arrow{r} \arrow{d}{\phi(f)} & M^\cc \arrow{r} \arrow[dashed]{d}{\bar{\Delta}(f)} & \Si \ca \lten_{\cb}M_\cb \lten_{\cb}\ca \arrow{d}{\Si \psi^{-1}(f)} \\
D\ca \arrow{r} & D(\ca \lten_{\cb} \ca) \arrow{r} & \Si D\cc \arrow{r} & \Si D\ca
\end{tikzcd}
\end{equation}
in $\cd(\ca^e)$ with the top row the triangle~(\ref{eq:triangle of resolutions 2}) and the bottom row the rotated dual of the triangle~(\ref{eq:triangle of resolutions 1}).
\begin{itemize}
\item[a)] The left square is commutative.
\item[b)] There exists a morphism $\bar{\Delta}(f)\colon M^\cc \to \Si D\cc$ in $\cd(\ca^e)$ such that the middle and right squares are commutative. Moreover, it is uniquely determined by the commutativity of the middle or the right square.
\end{itemize}
\end{proposition}

\begin{proof}
a) Since the square
\[
\begin{tikzcd}
D(M_\cb \lten_{\cb^e} \cb) \arrow[no head]{r}{\sim} \arrow[no head]{d}{\wr} & D((\ca \lten_{\cb}M_\cb \lten_{\cb} \ca)\lten_{\ca^e} \ca) \arrow{d}{\wr} \\
D(M\lten_{\ca^e}(\ca \lten_{\cb} \ca)) \arrow{r}{\sim} & D((\ca \lten_{\cb}M_\cb \lten_{\cb} \ca)\lten_{\ca^e}(\ca \lten_{\cb} \ca))
\end{tikzcd}
\]
in $\cd(R)$ is commutative, by adjunctions and taking homology, we obtain the commutative square
\[
\begin{tikzcd}
\Hom_{\cd(\cb^e)}(M_\cb, D\cb) \arrow[no head]{r}{\sim} \arrow[no head]{d}{\wr} & \Hom_{\cd(\ca^e)}(\ca \lten_{\cb}M_\cb \lten_{\cb} \ca, D\ca) \arrow{d}{\wr} \\
\Hom_{\cd(\ca^e)}(M, D(\ca \lten_{\cb} \ca)) \arrow{r}{\sim} & \Hom_{\cd(\ca^e)}(\ca \lten_{\cb}M_\cb \lten_{\cb} \ca, D(\ca \lten_{\cb} \ca))\mathrlap{\: .}
\end{tikzcd}
\]
Therefore, the images of $\phi(f)$ and $\psi^{-1}(f)$ in the lower-right corner coincide, which precisely means the desired commutativity.

b) By part~a) and the axiom of triangulated categories, there is a morphism $\bar{\Delta}(f)$ completing the diagram~(\ref{eq:morphism of triangles}) into a morphism of triangles in $\cd(\ca^e)$. It is uniquely determined by the commutativity of the middle square because the complex
\[
\begin{tikzcd}
\RHom_{\ca^e}(\ca \lten_{\cb}M_\cb \lten_{\cb}\ca, D\cc) \arrow[no head]{r}{\sim} & D((\ca \lten_{\cb}M_\cb \lten_{\cb}\ca)\lten_{\ca^e}\cc) \arrow[no head]{r}{\sim} & D(M_\cb \lten_{\cb^e}\cc)
\end{tikzcd}
\]
in $\cd(R)$ vanishes. It is uniquely determined by the commutativity of the right square because the complex
\[
\begin{tikzcd}
\RHom_{\ca^e}(M^\cc, D(\ca \lten_{\cb} \ca)) \arrow[no head]{r}{\sim} & D(M^\cc \lten_{\ca^e} (\ca \lten_{\cb} \ca)) \arrow[no head]{r}{\sim} & D(M^\cc \lten_{\cb^e} \cb)
\end{tikzcd}
\]
in $\cd(R)$ vanishes.
\end{proof}

Finally, since the sequence~(\ref{eq:exact sequence of dg categories}) of dg categories is exact, the dg functor $Q$ is a localization and hence so is $Q^e$. We denote the composed map
\[
\begin{tikzcd}
\Hom_{\cd(\cb^e)}(M_\cb, D\cb) \arrow{r}{\bar{\Delta}} & \Hom_{\cd(\ca^e)}(M^\cc, \Si D\cc) & \Hom_{\cd(\cc^e)}(M^\cc, \Si D\cc) \arrow[swap]{l}{\sim}
\end{tikzcd}
\]
by $\Delta$.

\subsection{Connecting morphisms and the dg enhancement of Amiot's construction}

We are now in the position to state the first main result of this article which gives a description of the connecting morphism as the construction in the previous section. This is an interpretation of the construction in \cite{HaniharaIyama22} in terms of dual Hochschild homology. We state the result under the setting that the dg $R$-categories $\ca$, $\cb$, and $\cc$ are flat as dg $R$-modules but it generalizes to the setting of arbitrary dg $R$-categories by replacing them with quasi-equivalent dg $R$-categories which are flat as dg $R$-modules.

\begin{theorem} \label{thm:connecting morphism}
Let $R$ be a commutative ring. Let $\ca$, $\cb$, and $\cc$ be small dg $R$-categories which are flat as dg $R$-modules and
\[
\begin{tikzcd}
0 \arrow{r} & \cb \arrow{r}{I} & \ca \arrow{r}{Q} & \cc \arrow{r} & 0
\end{tikzcd}
\]
an exact sequence of dg categories. Let $\nu$ be a dg autoequivalence of $\ca$ which maps $I\cb$ to itself. Let $M$ be a dg $\ca$-bimodule which is isomorphic to $\ca(\nu ?, -)$ in $\cd(\ca^e)$. Denote $\mathrm{R}I^e_* M$ by $M_\cb$ and $\mathrm{L}Q^{e*}M$ by $M^\cc$.
\begin{itemize}
\item[a)] The following square containing the connecting morphism $\delta$ and the map $\Delta$ in section~\ref{ss:a dg enhancement of Amiot's construction} is commutative.
\[
\begin{tikzcd}
H^0(D\HH(\cb, M_\cb)) \arrow{r}{\delta} \arrow[no head]{d}{\wr} & H^1(D\HH(\cc, M^\cc)) \arrow[no head]{d}{\wr} \\
\Hom_{\cd(\cb^e)}(M_\cb, D\cb) \arrow[swap]{r}{\Delta} & \Hom_{\cd(\cc^e)}(M^\cc, \Si D\cc)
\end{tikzcd}
\]
\item[b)] Suppose that the canonical morphism $\ca \to DD\ca$ in $\cd(\ca^e)$ is an isomorphism. If the canonical morphism $\ca \to \RHom_{\cb}(\ca, \ca)$ in $\cd(\ca^e)$ is an isomorphism, then the connecting morphism $\delta$ preserves non-degeneracy.
\end{itemize}
\end{theorem}

\begin{remarks} \label{rk:lift} \mbox{}
\begin{itemize}
\item[a)] The assignment $f \mapsto \Delta(f)$ is a lift of Amiot's construction~\cite{Amiot09} to the dg level. In fact, Remark~5.3.4 (based on Theorem~5.3.3) of~\cite{KellerLiu23b} shows that the connecting morphism in dual Hochschild homology is a lift of Amiot's construction to the dg level. The statement follows by combining this result with part~a) of Theorem~\ref{thm:connecting morphism}.
\item[b)] Part~b) of Theorem~\ref{thm:connecting morphism} is a lift of Theorem~1.3 of~\cite{Amiot09} to the dg level. In fact, the condition that the canonical morphism $\ca \to \RHom_{\cb}(\ca, \ca)$ is an isomorphism is a suitable dg version of the `local cover' condition, \cf Lemma~\ref{lem:local cover 1}.
\item[c)] We also give an equivalent condition at the dg level for the canonical morphism \linebreak $\ca \to \RHom_{\cb}(\ca, \ca)$ being an isomorphism, which means that the dg category $\ca$ is required to be no larger than the localizing completion of $\cb$, \cf Lemma~\ref{lem:local cover 2}.
\end{itemize}
\end{remarks}

We prepare the following lemma.

\begin{lemma} \label{lem:isomorphisms preserving}
Under the assumptions of part~b) of Theorem~\ref{thm:connecting morphism}, the bijections $\phi$ and $\psi^{-1}$ in Lemma~\ref{lem:adjunctions} preserves isomorphisms.
\end{lemma}
\begin{proof}
Let $f \colon M_\cb \to D\cb$ be an isomorphism in $\cd(\cb^e)$. We first prove that the morphism $\phi(f) \colon M \to D(\ca \lten_{\cb} \ca)$ in $\cd(\ca^e)$ is an isomorphism. By Lemma~\ref{lem:adjunctions}, it suffices to show that the morphism $\eta_M \colon M \to \RHom_{\cb^e}(\ca^e, M_\cb)$, determined by the unit of the adjunction $(? \lten_{\ca^e}\ca^e, \RHom_{\cb^e}(\ca^e, ?))$, is an isomorphism. Since the dg $\ca$-bimodule $M$ is isomorphic to $\ca(\nu ?, -)$ in $\cd(\ca^e)$ and the canonical morphism $\ca \to \RHom_{\cb}(\ca, \ca)$ in $\cd(\ca^e)$ is an isomorphism, so is the canonical morphism $M \to \RHom_{\cb}(\ca, M)$ in $\cd(\ca^e)$. Moreover, the object $\RHom_{\cb^e}(M\ten_R \ca, M_\cb)$ is isomorphic to $\RHom_{\cb^e}(\ca^e, \cb)$ in $\cd(\ca \ten_R \cb\op)$. By diagram chasing, the diagram
\begin{equation} \label{eq:unit diagram}
\begin{tikzcd}
M \arrow{rrr}{\eta_M} \arrow{d}{\wr} & & & _{\cb^e}(\ca^e, M_\cb) \arrow{dd}{_{\cb^e}(\ca^e, f)}[swap]{\wr} \\
_{\cb}(\ca, M) \arrow{d}{\wr} & & & \\
_{\cb}(\ca, DDM) & & & _{\cb^e}(\ca^e, D\cb) \\
_{\cb}(\ca, D_{\cb^e}(M\ten_R A, D\cb)) \arrow[swap]{u}{\wr}\arrow{d}{\wr}[swap]{_{\cb}(\ca, D_{\cb^e}(M\ten_R \ca, f))} & & & \\
_{\cb}(\ca, D_{\cb^e}(M\ten_R A, M_\cb)) \arrow[no head]{r}{\sim} & _{\cb}(\ca, D_{\cb^e}(\ca^e, \cb)) \arrow[swap]{rr}{_{\cb}(\ca, D\eta_\ca)} & & _{\cb}(\ca, D\ca) \arrow{uu}{\wr}
\end{tikzcd}
\end{equation}
in $\cd(\ca^e)$ is commutative. In this diagram, we write $_{\cb}(?,-)$ for $\RHom_{\cb}(?,-)$ and similarly for $_{\cb^e}(?,-)$.

We claim that the restriction of the morphism $\eta_M$ to $\ca \ten_R \cb \op$ is an isomorphism. Since the restriction of the dg $\ca^e$-module $\ca$ to $\cb^e$ is isomorphic to $\cb$ in $\cd(\cb^e)$, the restriction of the morphism $\RHom_{\cb}(\ca, D\eta_\ca)$ to $\ca \ten_R \cb \op$ is an isomorphism. By the commutativity of the restriction of the diagram~(\ref{eq:unit diagram}) to $\ca \ten_R \cb \op$, the restriction of the morphism $\eta_M$ to $\ca \ten_R \cb \op$ is an isomorphism. Since the dg $\ca$-bimodule $M$ is isomorphic to $\ca(\nu ?, -)$ in $\cd(\ca^e)$, this claim shows that the morphism $D\eta_\ca$ in $\cd(\cb \ten_R \ca \op)$ is an isomorphism and hence so is the morphism $\RHom_{\cb}(\ca, D\eta_\ca)$ in $\cd(\ca^e)$. Then by the commutativity of the diagram~(\ref{eq:unit diagram}), the morphism $\eta_M$ is an isomorphism. We conclude that the morphism $\phi(f)$ is an isomorphism.

We now prove that the morphism $\psi^{-1}(f) \colon \ca \lten_{\cb}M_\cb \lten_{\cb}\ca \to D\ca$ is an isomorphism. By Lemma~\ref{lem:adjunctions}, it suffices to show that the morphism $\eps_{D\ca} \colon \ca \lten_\cb D\cb \lten_\cb \ca \to D\ca$, determined by the counit of the adjunction $(?\lten_{\cb^e}\ca^e, \RHom_{\ca^e}(\ca^e, ?))$, is an isomorphism. Since the dg $\ca$-bimodule $M$ is isomorphic to $\ca(\nu ?, -)$ in $\cd(\ca^e)$, the object $\RHom_\cb(M, M)$ is isomorphic to $\RHom_\cb(\ca, \ca)$ in $\cd(\ca^e)$ and the object $\ca \lten_\cb M_\cb$ is isomorphic to $M$ in $\cd(\cb \ten \ca\op)$. By diagram chasing, the diagram
\begin{equation} \label{eq:counit diagram}
\begin{tikzcd}
\ca \arrow[swap]{rr}{\sim} \arrow{d}{\wr} & & DD\ca \arrow{rr}{D\eps_{D\ca}} & & D(D\cb \lten_{\cb^e}\ca^e) \arrow{d}{D(f\lten_{\cb^e}\ca^e)}[swap]{\wr} \\
_\cb(\ca, \ca) \arrow[no head]{d}{\wr} & & & & D(M_\cb \lten_{\cb^e}\ca^e) \arrow[no head, swap]{d}{\wr} \\
_\cb(M, M) \arrow[no head]{rr}{\sim}[swap]{_\cb(M, \ca \lten_{\cb}f)} & & _\cb(M, \ca \lten_\cb D\cb) \arrow[no head, swap]{rr}{_\cb(M, \eps_{D\ca})} & & _\cb(M, D\ca)
\end{tikzcd}
\end{equation}
in $\cd(\ca^e)$ is commutative. In this diagram, we write $_{\cb}(?,-)$ for $\RHom_{\cb}(?,-)$.

We claim that the restriction of the morphism $D\eps_{D\ca}$ to $\ca \ten_R \cb \op$ is an isomorphism. Since the restriction of the dg $\ca^e$-module $\ca$ to $\cb^e$ is isomorphic to $\cb$ in $\cd(\cb^e)$, the restriction of the morphism $\RHom_{\cb}(M, \eps_{D\ca})$ to $\ca \ten_R \cb \op$ is an isomorphism. By the commutativity of the restriction of the diagram~(\ref{eq:counit diagram}) to $\ca \ten_R \cb \op$, the restriction of the morphism $D\eps_{D\ca}$ to $\ca \ten_R \cb \op$ is an isomorphism. Now this claim shows that the morphism $\eps_{D\ca}$ in $\cd(\cb \ten_R \ca \op)$ is an isomorphism and hence so is the morphism $\RHom_{\cb}(M, \eps_{D\ca})$ in $\cd(\ca^e)$. Then by the commutativity of the diagram~(\ref{eq:counit diagram}), the morphism $D\eps_{D\ca}$ is an isomorphism and hence so is $\eps_{D\ca}$. We conclude that the morphism $\psi^{-1}(f)$ is an isomorphism.
\end{proof}

\begin{proof}[Proof of Theorem~\ref{thm:connecting morphism}]
a) By part~b) of Proposition~\ref{prop:morphism of triangles}, the map $\bar{\Delta}$ is determined by the property that $\bar{\Delta}(f)$ makes the middle and right squares of the diagram~(\ref{eq:morphism of triangles}) commutative for all $f\colon M_\cb \to D\cb$ in $\cd(\cb^e)$. Therefore, the map $\Delta$ fits into the following commutative diagram.
\[
\begin{tikzcd}
_{\cb^e}(M_\cb, D\cb) \arrow[no head]{r}{\sim} \arrow[no head]{d}{\wr} \arrow{ddrr}{\Delta} & _{\ca^e}(\ca \lten_{\cb}M_\cb \lten_{\cb} \ca, D\ca) \arrow{dr} & \\
_{\ca^e}(M, D(\ca \lten_{\cb} \ca)) \arrow{dr} & & _{\ca^e}(\Si^{-1}M^\cc, D\ca) \\
 & _{\ca^e}(M, \Si D\cc) & _{\cc^e}(M^\cc, \Si D\cc) \arrow[swap]{u}{\wr} \arrow[swap]{l}{\sim}
\end{tikzcd}
\]
In this diagram, we write $_{\ca^e}(?,-)$ for $\Hom_{\cd(\ca^e)}(?,-)$ and similarly for $_{\cb^e}(?,-)$ and $_{\cc^e}(?,-)$. Via adjunctions, this diagram is isomorphic to the following one.
\[
\adjustbox{max width=\textwidth}{
\begin{tikzcd}
H^0(D(M_\cb \lten_{\cb^e} \cb)) \arrow[no head]{r}{\sim} \arrow[no head]{d}{\wr} \arrow{ddrr}{\delta} & H^0(D((\ca \lten_{\cb}M_\cb \lten_{\cb} \ca)\lten_{\ca^e}\ca)) \arrow{dr} & \\
H^0(D(M \lten_{\ca^e}(\ca \lten_{\cb} \ca))) \arrow{dr} & & H^1(D(M^\cc\lten_{\ca^e} \ca)) \\
 & H^1(D(M \lten_{\ca^e} \cc)) & H^1(D(M^\cc \lten_{\cc^e} \cc)) \arrow[swap]{u}{\wr} \arrow[swap]{l}{\sim}
\end{tikzcd}
}
\]
This shows that the connecting morphism $\delta$ and the map $\Delta$ are compatible via adjunctions.

b) This is an immediate consequence of part~a), Lemma~\ref{lem:isomorphisms preserving}, and Proposition~\ref{prop:morphism of triangles}.
\end{proof}

We give some remarks concerning the local cover condition.

\begin{lemma}[\cite{HaniharaKeller25}] \label{lem:local cover 1}
Let $\ca$ and $\cb$ be pretriangulated dg categories and $I\colon \cb \to \ca$ a quasi-fully faithful dg functor. Suppose that for any objects $X$ and $Y$ in $H^0(\ca)$, we have a local $H^0(I)$-cover of $X$ relative to $Y$ in the sense of Definition~1.2 of \cite{Amiot09}. Then the canonical morphism $\ca \to \RHom_{\cb}(\ca, \ca)$ in $\cd(\ca^e)$ is an isomorphism.
\end{lemma}

\begin{proof}
For an object $X$ in $H^0(\ca)$, we write $X^\wg$ for the dg $\ca$-module $\ca(?, X)$. It suffices to show that for any objects $X$ and $Y$ in $H^0(\ca)$ and integer $p$, the induced map
\[
\begin{tikzcd}
\Hom_{\cd(\ca)}(X^\wg, (\Si^p Y)^\wg) \arrow{r} & \Hom_{\cd(\cb)}(X^\wg, (\Si^p Y)^\wg)
\end{tikzcd}
\]
in homology is bijective. Let $H^0(I)B_0 \to X$ be a local $H^0(I)$-cover of $X$ relative to $\Si^p Y$. We complete it to the triangle
\[
\begin{tikzcd}
X_1 \arrow{r} & H^0(I)B_0 \arrow{r} & X \arrow{r} & \Si X_1
\end{tikzcd}
\]
in $H^0(\ca)$. Let $H^0(I)B_1 \to X_1$ be a local $H^0(I)$-cover of $X_1$ relative to $\Si^p Y$. Then we have the exact sequence
\[
\begin{tikzcd}
0 \arrow{r} & H^0(\ca)(X, \Si^p Y) \arrow{r} & H^0(\ca)(H^0(I)B_0, \Si^p Y) \arrow{r} & H^0(\ca)(H^0(I)B_1, \Si^p Y) \: .
\end{tikzcd}
\]
It yields the following commutative diagram
\[
\begin{tikzcd}
0 \arrow{r} & _{\ca}(X^\wg, (\Si^p Y)^\wg) \arrow{r} \arrow{d} & _{\ca}((H^0(I)B_0)^\wg, (\Si^p Y)^\wg) \arrow{r} \arrow{d}{\wr} & _{\ca}((H^0(I)B_1)^\wg, (\Si^p Y)^\wg) \arrow{d}{\wr} \\
0 \arrow{r} & _{\cb}(X^\wg, (\Si^p Y)^\wg) \arrow{r} & _{\cb}(B_0^\wg, (\Si^p Y)^\wg) \arrow{r} & _{\cb}(B_1^\wg, (\Si^p Y)^\wg)
\end{tikzcd}
\]
with exact rows. In this diagram, we write $_{\ca}(?,-)$ for $\Hom_{\cd(\ca)}(?,-)$ and similarly for $_{\cb}(?,-)$. Since the middle and right vertical maps are bijective, by the Five Lemma, so is the left one.
\end{proof}

\begin{lemma} \label{lem:local cover 2}
Let $\ca$ and $\cb$ be dg categories and $I\colon \cb \to \ca$ a dg functor. Then the following are equivalent.
\begin{itemize}
\item[i)] The dg functor $I$ is quasi-fully faithful and the the restriction functor
\[
\mathrm{R}I_* \colon \per \ca \to \cd(\cb)
\]
is fully faithful.
\item[ii)] There exists a quasi-fully faithful dg functor $J\colon \ca \to \cd_{dg}(\cb)$ such that the composition $J\circ I$ is the dg Yoneda functor.
\end{itemize}
\end{lemma}

\begin{proof}
Let us prove the implication from i) to ii). Let $J$ be the composed dg functor
\[
\begin{tikzcd}
\ca \arrow{r} & \per\!_{dg}\ca \arrow{r} & \cd_{dg}(\cb) \: ,
\end{tikzcd}
\]
where the first dg functor is determined by the dg Yoneda functor and the second dg functor is determined by restriction. Since the restriction functor $\mathrm{R}I_* \colon \per \ca \to \cd(\cb)$ is fully faithful, the dg functor $J$ is quasi-fully faithful. Because the dg functor $I$ is quasi-fully faithful, the composed dg functor $J\circ I$ is naturally isomorphic to the dg Yoneda functor. This implies the statement.

We now prove the implication from ii) to i). The quasi-full faithfulness of the dg functor $I$ follows from that of $J\circ I$ and $J$. Since the dg functor $J$ is quasi-fully faithful, the unit \linebreak $\id_{\cd(\ca)} \iso \mathrm{R}J_* \circ \mathrm{L}J^*$ is a natural isomorphism. So we have the composed natural isomorphism
\[
\begin{tikzcd}
\mathrm{R}I_* \arrow{r}{\sim} & \mathrm{R}I_* \circ \mathrm{R}J_* \circ \mathrm{L}J^* & \mathrm{R}(J \circ I)_* \circ \mathrm{L}J^* \arrow[swap]{l}{\sim} \: .
\end{tikzcd}
\]
We have the commutative square
\[
\begin{tikzcd}
H^0(\ca) \arrow{r}{H^0(J)} \arrow{d} & \cd(\cb) \arrow{d}{\wr} \\
\per \ca \arrow[swap]{r}{\mathrm{L}J^*} & \per(\cd_{dg}(\cb)) \mathrlap{\: ,}
\end{tikzcd}
\]
where the vertical functors are induced by the dg Yoneda functors. Since the functor $\mathrm{R}(J \circ I)_*$ is the quasi-inverse of the functor $\cd(\cb) \iso \per (\cd_{dg}(\cb))$ and $\per \ca$ is the thick subcategory generated by the image of the functor $H^0(\ca) \to \per \ca$, the full faithfulness of the restriction functor $\mathrm{R}I_* \colon \per \ca \to \cd(\cb)$ follows from that of the functor $H^0(J)$.
\end{proof}

\section{Calabi--Yau structures on bounded derived categories and singularity categories}

We construct Calabi--Yau structures on derived and singularity categories of symmetric $R$-orders. This is done by first constructing Calabi--Yau structures over the commutative base ring $R$, and then transferring them to Calabi--Yau structures over the base field $k$ via the `base change' formulas formulated in Propositions~\ref{prop:right base change} and \ref{prop:left base change}. It is therefore essential to distinguish the base rings explicitly. For a commutative ring $R$ and a dg $R$-category $\ca$, we write $D_R$ for the functor $\RHom_R(?, R)$ and $\ca^e_R$ for the dg category $\ca \ten_R \ca \op$. We write $\HH(\ca /R)$ and $\HC(\ca /R)$ for the Hochschild complex respectively the cyclic complex over $R$ of $\ca$. Unadorned above notation is over $k$.

\subsection{Right Calabi--Yau structures on $\sg\!_{dg}\La$}

We now study the dg singularity categories associated with symmetric orders. The first observation is standard.

\begin{lemma} \label{lem:bijectivity}
Let $R$ be a commutative ring and $A$ a connective dg $R$-algebra. Then the canonical map
\[
H^p(D_R \HC(A/R)) \longrightarrow H^p(D_R \HH(A/R))
\]
is bijective for all non-positive integers $p$.
\end{lemma}

\begin{proof}
For simplicity, we omit the $R$ in the adorned notation to write $D$, $\HH$, and $\HC$. If we apply the triangle functor $D=\RHom_R(?, R)$ to the ISB triangle
\[
\begin{tikzcd}
  \HH(A)\arrow{r}{I} & \HC(A)\arrow{r}{S} & \Si^2 \HC(A)\arrow{r}{B} & \Si \HH(A)
\end{tikzcd}
\]
and then take homology, we obtain the long exact sequence
\[
\begin{tikzcd}
H^{p-2}(D\HC(A)) \arrow{r} & H^p(D\HC(A)) \arrow{r} & H^p(D\HH(A)) \arrow{r} & H^{p-1}(D\HC(A)) \: .
\end{tikzcd}
\]
Since the dg algebra $A$ is connective, the homology of the complex $D\HC(A)$ vanishes in all negative degrees. Then this implies the statement.
\end{proof}

Recall that a (not necessarily commutative) noetherian ring is {\em Iwanaga--Gorenstein} if it has finite injective dimension as both a left and a right module over itself. For an Iwanaga--Gorenstein ring $\La$, the singularity category $\sg \La$ is triangle equivalent to the stable category of the category $\Gproj \La$ of finitely generated Gorenstein projective $\La$-modules. The above lemma yields the following refinement of Theorem~3.3 of~\cite{HaniharaIyama22}.

\begin{proposition} \label{prop:right CY structure}
Let $R$ be a commutative Gorenstein ring and $\La$ a symmetric $R$-order. Then the dg category $\sg\!_{dg}\La$ carries a right $(-1)$-Calabi--Yau structure over $R$.
\end{proposition}

\begin{proof}
For simplicity, we omit the $R$ in the adorned notation to write $D$, $\HH$, and $\HC$. Denote the dg categories $\cd^b_{dg}(\mod \La)$, $\per\!_{dg}\La$, $\sg\!_{dg}\La$ by $\ca$, $\cb$, $\cc$, respectively. After possibly replacing them with quasi-equivalent dg $R$-categories, we may and will assume that they are flat as dg $R$-modules.

Since $\La$ is a symmetric $R$-order, we have a non-degenerate class in $H^0(D\HH(\La))$. By Lemma~\ref{lem:bijectivity}, it lifts uniquely to the class in $H^0(D\HC(\La))$ which is a right $0$-Calabi--Yau structure over $R$ on $\La$. Since the dg category $\cb$ is derived Morita equivalent to the algebra $\La$, we obtain a right $0$-Calabi--Yau structure over $R$ on $\cb$, which is denoted by $[\tilde{x}]$.
The exact sequence
\[
\begin{tikzcd}
0 \arrow{r} & \cb \arrow{r} & \ca \arrow{r} & \cc \arrow{r} & 0
\end{tikzcd}
\]
of dg categories yields a commutative square
\[
\begin{tikzcd}
H^0(D\HC(\cb)) \arrow{r}{\tilde{\delta}} \arrow{d} & H^1(D\HC(\cc)) \arrow{d} \\
H^0(D\HH(\cb)) \arrow[swap]{r}{\delta} & H^1(D\HH(\cc))
\end{tikzcd}
\]
with the horizontal maps the connecting morphisms. We claim that the class $\tilde{\delta}([\tilde{x}])$ is a right $(-1)$-Calabi--Yau structure over $R$ on $\cc$. Denote the underlying Hochschild class of $[\tilde{x}]$ by $[x]$. It suffices to show that the class $\delta([x])$ is non-degenerate. To prove this, we verify the conditions in part~b) of Theorem~\ref{thm:connecting morphism}. It is well-known that the canonical morphism $\ca \to DD\ca$ in $\cd(\ca^e)$ is an isomorphism, \cf the implication from (iv) to (i) in Proposition~V.2.1 of~\cite{Hartshorne66}. On the other hand, by the implication from ii) to i) in Lemma~\ref{lem:local cover 2}, the canonical morphism $\ca \to \RHom_{\cb}(\ca, \ca)$ in $\cd(\ca^e)$ is an isomorphism. Therefore, we apply part~b) of Theorem~\ref{thm:connecting morphism} for $M=\ca$, the class $\delta([x])$ is non-degenerate. We conclude that the class $\tilde{\delta}([\tilde{x}])$ is a right $(-1)$-Calabi--Yau structure over $R$ on $\cc$.
\end{proof}

Next we relate this Calabi--Yau structure over $R$ to that over $k$. Recall that the {\em support} $\Supp M$ of a module $M$ over a commutative ring $R$ is defined as the set of prime ideals of $R$ such that the the localizations of $M$ at them do not vanish. For a dg $R$-category $\cc$, we define its {\em support} $\Supp \cc$ as the union of $\Supp H^p(\cc(X,Y))$, where $X$ and $Y$ run through the objects in $\cc$ and $p$ through the integers. We first note that $R$-linear invariants of a dg $R$-category are supported on its support.

\begin{lemma} \label{lem:support}
Let $R$ be a commutative ring and $\cc$ a dg $R$-category. Then the Hochschild homology $\HH_p(\cc/R)$ and cyclic homology $\HC_p(\cc/R)$ are supported on $\Supp \cc$ for all integers $p$. In particular, if the homology $H^p(\cc(X, Y))$ is of finite length as an $R$-module for all objects $X$ and $Y$ in $\cc$ and integers $p$, then $\HH_p(\cc/R)$ and $\HC_p(\cc/R)$ are supported on maximal ideals for all integers $p$.
\end{lemma}

\begin{proof}
After possibly replacing $\cc$ with a quasi-equivalent dg $R$-category, we may and will assume that it is flat as a dg $R$-module. For any prime ideal $\mathfrak p$ of $R$, since the localization at $\mathfrak p$ is an exact functor, we have the isomorphisms
\[
\begin{tikzcd}
\HH_*(\cc /R)\ten_R R_{\mathfrak p} \arrow[no head]{r}{\sim} & H^{-*}(\HH(\cc /R)\ten_R R_{\mathfrak p}) \arrow[no head]{r}{\sim} & \HH_*(\cc_{\mathfrak p}/R_{\mathfrak p}) \: .
\end{tikzcd}
\]
Similarly for cyclic homology. Then the first statement is clear. The second statement follows from the fact that $R$-modules of finite length are supported on maximal ideals.
\end{proof}

For a commutative $k$-algebra $R$, we denote by $\Max^k_d R$ the set of maximal ideals of $R$ of height at least $d$ whose residue fields are finite-dimensional over $k$. We now state an important relationship between right Calabi--Yau structures over $R$ and right Calabi--Yau structures over $k$.

\begin{proposition} \label{prop:right base change}
Let $R$ be a commutative Gorenstein $k$-algebra of Krull dimension $d$ and $\cc$ a dg $R$-category satisfying $\Supp \cc \subseteq \Max^k_d R$. Then we have a commutative square
\[
\begin{tikzcd}
D_R \HC(\cc /R) \arrow[no head]{r}{\sim} \arrow{d} & \Si^{-d}D\HC(\cc) \arrow{d} \\
D_R \HH(\cc /R) \arrow[no head]{r}{\sim} & \Si^{-d}D\HH(\cc)
\end{tikzcd}
\]
in $\cd(k)$. Moreover, for any integer $p$, the induced bijection
\[
H^{-p}(D_R \HH(\cc /R)) \xlongrightarrow{_\sim} H^{-p-d}(D\HH(\cc))
\]
in homology preserves and detects non-degeneracy, and the induced bijection
\[
H^{-p}(D_R \HC(\cc /R)) \xlongrightarrow{_\sim} H^{-p-d}(D\HC(\cc))
\]
in homology restricts to the bijection between right $p$-Calabi--Yau structures over $R$ and right $(p+d)$-Calabi--Yau structures over $k$.
\end{proposition}

\begin{proof}
Since the support of $\cc$ is contained in $\Max^k_d R$, by Lemma~\ref{lem:support}, the same holds for the support of $\HH(\cc/R)$ and $\HC(\cc/R)$. Put $E=\bigoplus_{\mathfrak{m} \in \Max^k_d R}E_R(R/\mathfrak{m})$, where $E_R(R/\mathfrak{m})$ denotes the injective envelope of the $R$-module $R/\mathfrak{m}$. Then we have the isomorphisms
\[
\begin{tikzcd}
D_R \HH(\cc /R) \arrow[no head]{r}{\sim} & \Si^{-d}\RHom_R(\HH(\cc /R),E) \arrow[no head]{r}{\sim} & \Si^{-d}D\HH(\cc /R)
\end{tikzcd}
\]
in $\cd(k)$, \cf Example~1 in page~63 of~\cite{Hartshorne67} for the second isomorphism. Similarly, we have the isomorphism $D_R \HC(\cc /R) \simeq \Si^{-d}D\HC(\cc /R)$ in $\cd(k)$. Then the statements follow from the fact that a morphism $\cc \to \Si^{-p-d}D\cc$ in $\cd(\cc^e_R)$ is an isomorphism if and only if its underlying morphism in $\cd(\cc^e)$ is an isomorphism for all integers $p$.
\end{proof}

Now the preceding discussions lead to the following result.
Recall that the {\em singular locus} $\Sing_R \La$ of an $R$-algebra $\La$ is defined as the set of prime ideals of $R$ such that the localizations of $\La$ at them have infinite global dimensions.

\begin{theorem} \label{thm:right CY structure}
Let $k$ be a field and $R$ a commutative Gorenstein $k$-algebra of Krull dimension $d$. Let $\La$ be a symmetric $R$-order satisfying $\Sing_R \La \subseteq \Max_d^k R$. Then the dg category $\sg\!_{dg}\La$ carries a right $(d-1)$-Calabi--Yau structure over $k$.
\end{theorem}

\begin{proof}
Denote the dg category $\sg\!_{dg}\La$ by $\cc$. By Proposition~\ref{prop:right CY structure}, it carries a right \linebreak $(-1)$-Calabi--Yau structure over $R$. Since the dg category $\cc_{\mathfrak p}$ is derived Morita equivalent to $\sg\!_{dg}\La_{\mathfrak{p}}$ for all prime ideals $\mathfrak{p}$ of $R$, we have $\Supp \cc \subseteq \Sing_R \La$. By the assumption $\Sing_R \La \subseteq \Max_d^k R$, this implies that we have $\Supp \cc \subseteq \Max_d^k R$. Therefore, by Proposition~\ref{prop:right base change}, the right $(-1)$-Calabi--Yau structure over $R$ gives rise to a right $(d-1)$-Calabi--Yau structure over $k$.
\end{proof}

As an application of the right $(d-1)$-Calabi--Yau structure constructed as above, the following result establishes a connection between singularity categories and cluster categories. This generalizes (respectively, refines) many previous results \cite{KellerReiten08, ThanhofferVandenBergh16, AmiotIyamaReiten15, KalckYang18, Hanihara22, Liu26} in the sense that our result shows $\Pi$ to be a deformed dg preprojective algebra.

\begin{corollary} \label{cor:right CY structure}
Let $k$ be a field of characteristic $0$ and $(R, \mathfrak{m})$ a complete commutative Gorenstein local $k$-algebra of Krull dimension $d\geq 2$. Let $\La$ be a symmetric $R$-order satisfying $\Sing_R \La \subseteq \{\mathfrak{m}\}$. Suppose that the triangulated category $\sg \La$ contains a \mbox{$(d-1)$-cluster}-tilting object. Then there exists a $d$-dimensional deformed dg preprojective algebra $\Pi$ such that $\sg \La$ is triangle equivalent to the associated cluster category $\cc_\Pi$.
\end{corollary}

\begin{proof}
Since the commutative ring $R$ is Gorenstein, the symmetric $R$-order $\La$ is Iwanaga--Gorenstein. Thus, the singularity category $\sg \La$ is triangle equivalent to the stable category of $\Gproj \La$. Since the morphism spaces of $\Gproj \La$ are finitely generated modules over the complete commutative local $k$-algebra $R$, the category $\Gproj \La$ is naturally enriched over the category of pseudo-compact vector spaces and the category $\sg\La$ is Karoubian. Then the statement is a consequence of Theorem~\ref{thm:right CY structure} and the implication from ii) to i) in Theorem~6.2.1 of~\cite{KellerLiu23b}.
\end{proof}

\subsection{Left Calabi--Yau structures on $\cd^b_{dg}(\mod \La)$}

We study the dg bounded derived categories of symmetric orders in this section.
We will often make some technical assumptions on the ground ring. A commutative noetherian ring $R$ is called {\em J-2} if the singular locus of any finitely generated commutative $R$-algebra is closed. Typical examples of J-2 rings are finitely generated commutative algebras over a field. We first note the `Koszul duality' of a module-finite $R$-algebra and its dg derived category. The following proposition is an $R$-linear version of Corollary~15.2 of \cite{VandenBergh15}.

\begin{proposition} \label{prop:Koszul duality}
Let $R$ be a commutative Gorenstein J-2 ring and $A$ a module-finite $R$-algebra. Then the dg $R$-category $\cd^b_{dg}(\mod A)$ is smooth and we have an isomorphism
\[
\begin{tikzcd}
D_R \HH(A/R) \arrow{r}{\sim} & \HH(\cd^b_{dg}(\mod A)/R)
\end{tikzcd}
\]
in $\cd(R)$. Moreover, the induced bijection in each homology preserves and detects non-degeneracy.
\end{proposition}

\begin{proof}
For simplicity, we omit the $R$ in the adorned notation to write $D$, $A^e$, and $\HH$. Since the commutative ring $R$ is J-2 and the $R$-algebra $A$ is module-finite, by Theorem~4.15 of \cite{ElaginLuntsSchnuerer20}, the triangulated category $\cd^b(\mod A)$ has a generator $M$. After possibly replacing $A$ with a quasi-isomorphic dg $R$-algebra, we may and will assume that it is flat as a dg $R$-module. By abuse of notation, we still write $\cd^b(\mod A)$ and $M$ for the corresponding subcategory respectively the corresponding object of $\cd(A)$ under the equivalence induced by the quasi-isomorphism. Let $B$ be a cofibrant replacement of the dg $R$-algebra $\RHom_{A}(M, M)$. Since the commutative ring $R$ is Gorenstein, we have the duality $D\colon \cd^b(\mod A)\op \to \cd^b(\mod A\op)$. It yields the isomorphism
\[
\RHom_{A^e}(M\lten_R DM, M\lten_R DM) \xlongleftarrow{_\sim} \RHom_A(M, M)\lten_R \RHom_{A\op}(DM, DM)=B^e \: .
\]
By part~(b) of Theorem 4.18 of \cite{ElaginLuntsSchnuerer20}, the object $M\lten_R DM$ is a generator of the triangulated category $\cd^b(\mod A^e)$. So we have the triangle equivalence
\[
F=\RHom_{A^e}(M\lten_R DM, ?) \colon \cd^b(\mod A^e) \xlongrightarrow{_\sim} \per B^e \: .
\]
Its quasi-inverse is given by $G=? \lten_{B^e}(M\lten_R DM)$. We claim that
\begin{itemize}
\item[a)] $F(DA)$ is isomorphic to $B$ in $\per B^e$,
\item[b)] $FA$ is isomorphic to $B^{\vee}$ in $\per B^e$.
\end{itemize}
Indeed, we have the isomorphisms
\[
F(DA)\simeq D((M\lten_R DM)\lten_{A^e}A)\liso D(M\lten_ADM) \liso \RHom_A(M, M) \simeq B
\]
in $\per B^e$, which proves part~a). In particular, the dg $R$-algebra $B$ is smooth, \cf Theorem~5.1 of \cite{ElaginLuntsSchnuerer20}. To prove part~b), we show that $G(B^\vee)$ is isomorphic to $A$ in $\cd^b(\mod A^e)$. By the smoothness of $B$, the canonical morphism $G(B^\vee) \to \RHom_{B^e}(B, M\lten_R DM)$ is an isomorphism. By combining with the canonical isomorphisms
\[
\RHom_{B^e}(B, M\lten_R DM) \iso \RHom_{B^e}(B, \RHom_R(DM, DM)) \iso \RHom_B(DM, DM)
\]
we obtain the isomorphism $G(B^\vee) \iso \RHom_B(DM, DM)$ in $\cd^b(\mod A^e)$. On the other hand, since $M$ is a generator of the triangulated category $\cd^b(\mod A)$, we have the triangle equivalence
\[
\RHom_A(M, ?) \colon \cd^b(\mod A) \xlongrightarrow{_\sim} \per B \: .
\]
It maps the object $DA$ to $DM$. So we have the composed isomorphism
\[
\begin{tikzcd}
A \arrow{r}{\sim} & \RHom_A(DA, DA) \arrow{r}{\sim} & \RHom_B(DM, DM)
\end{tikzcd}
\]
in $\cd^b(\mod A^e)$. We deduce that $G(B^\vee)$ is isomorphic to $A$ in $\cd^b(\mod A^e)$, which proves part~b). Now, by the full faithfulness of $F$, we have the isomorphism
\[
\begin{tikzcd}
\RHom_{A^e}(A,DA)\arrow{r}{\sim} & \RHom_{B^e}(B^\vee,B)
\end{tikzcd}
\]
in $\cd(R)$ and the induced bijection
\[
H^p(D\HH(A)) \xlongrightarrow{_\sim} H^p(\HH(B))
\]
in homology preserves and detects non-degeneracy for all integers $p$. Then the statement follows from the fact that the dg $R$-category $\cd^b_{dg}(\mod A)$ is quasi-equivalent to $\per\!_{dg}B$ and the dg $R$-category $\per\!_{dg}B$ is derived Morita equivalent to the dg $R$-algebra $B$. 
\end{proof}

The Koszul duality over $R$ gives the following observation.

\begin{proposition} \label{prop:left CY structure}
Let $R$ be a commutative Gorenstein J-2 ring and $\La$ a symmetric $R$-order. Then the dg category $\cd^b_{dg}(\mod \La)$ carries a left $0$-Calabi--Yau structure over $R$.
\end{proposition}

\begin{proof}
For simplicity, we omit the $R$ in the adorned notation to write $D$, $\HH$, and $\HN$. Denote the dg category $\cd^b_{dg}(\mod \La)$ by $\ca$. Since $\La$ is a symmetric $R$-order, it gives rise to a non-degenerate class $[x]$ in $H^0(D\HH(\La))$. By Proposition~\ref{prop:Koszul duality}, the dg $R$-category $\ca$ is smooth. So the class $[x]$ corresponds to a non-degenerate class $[\xi]$ in $\HH_0(\ca)$ and the Hochschild homology $\HH_p(\ca) \liso H^{-p}(D\HH(\La))$ vanishes for all positive integers $p$. Then by part~(1) of Lemma~5.10 of \cite{BravDyckerhoff19} (the proof applies to any commutative ground ring), the canonical map $\HN_0(\ca) \to \HH_0(\ca)$ is bijective. Therefore, the preimage of $[\xi]$ under this bijection gives rise to a left $0$-Calabi--Yau structure over $R$ on the dg category $\ca$.
\end{proof}

The following lemma and proposition are from \cite{HaniharaKeller25}. We include proofs for the convenience of the reader.

\begin{lemma}[\cite{HaniharaKeller25}] \label{lem:full faithfulness}
Let $A$ be a noetherian $k$-algebra. Denote the dg category $\cd^b_{dg}(\mod A)$ by $\ca$. Then the restriction functor $\per \ca^e \to \cd(A^e)$ is fully faithful and maps $\ca$ to $A$ and $\ca^\vee$ to $A^\vee$.
\end{lemma}

\begin{proof}
By definition, the restriction functor $\per \ca^e \to \cd(A^e)$ maps $\ca$ to $A$ and $\ca^\vee$ to
\[
\RHom_{\ca^e}(\ca, \ca(?, A)\ten_k \ca(A, -)) \: .
\]
By the implication from ii) to i) in Lemma~\ref{lem:local cover 2}, it is fully faithful and hence we have the isomorphism
\[
\begin{tikzcd}
\RHom_{\ca^e}(\ca, \ca(?, A)\ten_k \ca(A, -)) \arrow{r}{\sim} & \RHom_{A^e}(A, A^e)
\end{tikzcd}
\]
in $\cd(A^e)$. This ends the proof.
\end{proof}

We now have the following left Calabi--Yau version of Proposition~\ref{prop:right base change}.

\begin{proposition}[\cite{HaniharaKeller25}] \label{prop:left base change}
Let $R$ be a commutative left $d$-Calabi--Yau $k$-algebra and $\ca$ a smooth dg $R$-category. Then $\ca$ is smooth as a dg $k$-category and we have an isomorphism
\[
\HH(\ca/R) \simeq \Si^{-d}\HH(\ca)
\]
in $\cd(k)$. Moreover, the induced bijection
\[
\HH_p(\ca/R) \simeq \HH_{p+d}(\ca)
\]
in homology preserves and detects non-degeneracy for all integers $p$.
\end{proposition}

\begin{proof}
After possibly replacing $\ca$ with a quasi-equivalent dg $R$-category, we may and will assume that it is flat as a dg $R$-module. Since the dg $R$-category $\ca$ is smooth and the commutative ring $R$ is regular, the underlying dg $k$-category of $\ca$ is smooth. So we have the isomorphisms
\[
\HH(\ca/R)\xlongrightarrow{_\sim} \RHom_{\ca^e_R}(\RHom_{\ca^e_R}(\ca, \ca^e_R), \ca)
\]
in $\cd(R)$ and
\[
\HH(\ca)\xlongrightarrow{_\sim} \RHom_{\ca^e}(\ca^{\vee}, \ca)
\]
in $\cd(k)$. Since the commutative $k$-algebra $R$ is left $d$-Calabi--Yau, we have the isomorphisms
\[
\begin{tikzcd}
\ca^e_R & \Si^d \ca^e \lten_{R^e} R^\vee \arrow[swap]{l}{\sim} \arrow{r}{\sim} & \Si^d \RHom_{R^e}(R, \ca^e)
\end{tikzcd}
\]
in $\cd(\ca^e_R \ten_k (\ca^e)\op)$.
By adjunctions, we have the isomorphisms
\[
\begin{tikzcd}
\ca^\vee \arrow{r}{\sim} & \RHom_{\ca^e_R}(\ca, \RHom_{\ca^e}(\ca^e_R, \ca^e)) \arrow{r}{\sim} & \RHom_{\ca^e_R}(\ca, \RHom_{R^e}(R, \ca^e))
\end{tikzcd}
\]
in $\cd(\ca^e)$. Therefore, the dg $\ca$-bimodule $\RHom_{\ca^e_R}(\ca, \ca^e_R)$ is isomorphic to $\Si^d \ca^\vee$ in $\cd(\ca^e)$. We conclude that $\HH(\ca/R)$ is isomorphic to $\Si^{-d}\HH(\ca)$ in $\cd(k)$. The last statement follows from the fact that a morphism $\Si^p \RHom_{\ca^e_R}(\ca, \ca^e_R) \to \ca$ in $\cd(\ca^e_R)$ is an isomorphism if and only if its underlying morphism in $\cd(\ca^e)$ is an isomorphism for all integers $p$.
\end{proof}

Now the preceding observations yield the following result which a non-commutative analogue of Proposition~5.12 of \cite{BravDyckerhoff19}.

\begin{theorem} \label{thm:left CY structure}
Let $k$ be a perfect field and $R$ a finitely generated commutative Gorenstein $k$-algebra of Krull dimension $d$. Let $\La$ be a symmetric $R$-order. Then the dg category $\cd^b_{dg}(\mod \La)$ carries a left $d$-Calabi--Yau structure over $k$.
\end{theorem}

\begin{proof}
Denote the dg category $\cd^b_{dg}(\mod \La)$ by $\ca$. By Theorem~5.1 of~\cite{ElaginLuntsSchnuerer20}, it is smooth as a dg $k$-category. Since the commutative ring $R$ is Gorenstein, it is a symmetric order over its Noether normalization and hence so is $\La$. Therefore, after possibly replacing $R$ its Noether normalization, we may and will assume that it is the polynomial algebra in $d$ variables. In particular, it is left $d$-Calabi--Yau. By Proposition~\ref{prop:left CY structure}, the dg $R$-category $\ca$ is smooth and we have a non-degenerate class $[\xi]$ in $\HH_0(\ca/R)$. By Proposition~\ref{prop:left base change}, it yields a non-degenerate class $[\xi']$ in $\HH_d(\ca)$. By Lemma~\ref{lem:full faithfulness}, the Hochschild homology $\HH_p(\ca)\iso \Hom_{\cd(\La^e)}(\La, \Si^{d-p}\La)$ vanishes for all degrees $p>d$. Then by part~(1) of Lemma~5.10 of \cite{BravDyckerhoff19}, the canonical map $\HN_d(\ca) \to \HH_d(\ca)$ is bijective. Therefore, the preimage of $[\xi']$ under this bijection gives rise to a left $d$-Calabi--Yau structure over $k$ on the dg category $\cd^b_{dg}(\mod \La)$.
\end{proof}

\begin{corollary} \label{cor:left CY structure}
Under the assumption of Theorem~\ref{thm:left CY structure}, the dg category $\sg\!_{dg}\La$ carries a left $d$-Calabi--Yau structure over $k$.
\end{corollary}

\begin{proof}
Since the dg category $\sg\!_{dg}\La$ is the dg quotient of $\cd^b_{dg}(\mod \La)$ by $\per\!_{dg}\La$, the statement follows from Theorem~\ref{thm:left CY structure} and part~(d) of Propsoition~3.10 of \cite{Keller11b}.
\end{proof}



\def\cprime{$'$} \def\cprime{$'$}
\providecommand{\bysame}{\leavevmode\hbox to3em{\hrulefill}\thinspace}
\providecommand{\MR}{\relax\ifhmode\unskip\space\fi MR }
\providecommand{\MRhref}[2]{%
  \href{http://www.ams.org/mathscinet-getitem?mr=#1}{#2}
}
\providecommand{\href}[2]{#2}

\end{document}